\documentclass{article}

\usepackage[T1]{fontenc}
\usepackage[utf8]{inputenc}
\usepackage[english]{babel}
\usepackage{amsmath,amssymb,amsthm,mathtools}
\mathtoolsset{showonlyrefs=true}

\usepackage{changepage}

\usepackage{atbegshi}
\usepackage{url}
\usepackage{graphicx}
\usepackage{pstricks}
\usepackage{multido}
\usepackage{placeins}

\AtBeginShipout{%
  \setbox0=\box\AtBeginShipoutBox
  \dimen0=\ht0
  \advance\dimen0 by \dp0
  \setbox\AtBeginShipoutBox=\hbox{%
    \kern-0.125\wd0
    \vbox{%
      \vskip-0.135\dimen0
      \hbox{\scalebox{1.25}{\box0}}%
    }%
  }%
}

\numberwithin{equation}{section}

\newtheorem{theorem}{Theorem}[section]
\newtheorem{lemma}[theorem]{Lemma}

\theoremstyle{definition}
\newtheorem{definition}[theorem]{Definition}

\newtheorem{corollary}[theorem]{Corollary}

\newcommand{\Ptwo}{P\otimes P}

\newcommand{\Aboxk}[3]{#1\,\boxtimes_{#2}#3}

\newcommand{\Aocc}[2]{#1\circ #2}

\newcommand{\setcoord}[3]{#1(#2\leftarrow #3)}

\def\zu{\{\,0,1\,\}}
\def\zuen{\smash{\{\,0,1\,\}^n}}

\def\unn{\{\,1,\dots,n\,\}}
\def\zun{\{\,0,\dots,n\,\}}

\begin{document}

%\title{A quantitative estimate of the BK gap}
\title{The case of equality in BK }

\author{Rapha\"el Cerf\thanks{Universit\'e Paris-Saclay, CNRS, Laboratoire de math\'ematiques d'Orsay, 91405 Orsay.}
 \and Pierre Tesio\footnotemark[1]}
%\date{}
\maketitle

\begin{abstract}
  We characterize the pairs of increasing events $A,B$ for which there is
  equa\-li\-ty in the BK inequality. Namely, we show that 
    $$P(A\circ B)=P(A)P(B)$$ if and only if
  all the configurations in $A\times B$ admit disjoint witnesses 
  for $A$ and $B$.
  %, but none of them  has a symmetric disjoint pivot.
    %We state a strengthened BK inequality. 
    We discuss the strengthened BK inequality, and we
    provide a new simplified proof of this inequality.
  %  Finally,
    %we obtain partial results on the BK gap
  %$$\Delta_{BK}(A,B) = P(A)P(B)-P(A\circ B)\,,$$
  %and we propose a quantitative lower bound of it. It 
  %involves terms which were already present in the strengthened 
  %FKG inequality of Talagrand.
\end{abstract}

\section{Introduction}

We fix an integer \(n\geq 1\) and a real number \(p\) in \(]0,1[\).
We endow the hypercube \(\Omega=\{\,0,1\,\}^n\) with the Bernoulli product measure \(P\) of parameter~\(p\).
%We fix an integer \(n\geq 1\), and for convenience, we define
%\(\Omega=\{\,0,1\,\}^n\).
%%$$\Omega\,=\,\{\,0,1\,\}^n\,.$$
%A generic element of \(\zuen\) is denoted by \(x=(x_1,\dots,x_n)\).
%We fix a parameter \(p\) in \(]0,1[\) and
%we endow \(\Omega\) with the Bernoulli product measure \(P\) of parameter~\(p\).
%\smallskip
%
%\noindent
%{\bf Order.}
%We endow $\Omega$ with the product order:
An element \(x=(x_1,\dots,x_n)\) of $\Omega$ is smaller than or equal to
an element \(y=(y_1,\dots,y_n)\) of \(\Omega\)
if \(x_i\leq y_i\) for \(1\leq i\leq n\).
A subset \(A\) of \(\Omega\) is increasing if the map
\(x\in\Omega\mapsto 1_A(x)\) is non-decreasing.
%for the product order.
\smallskip

Besides the FKG inequality, the most significant correlation inequality
on the space $(\Omega,P)$
is the Van den Berg-Kesten inequality \cite{BK},
called in short the BK inequality. It says that, for any 
%pair of 
increasing
events $A,B$ included in $\Omega$, we have
\begin{equation}
  \label{bkineq}
  P(A\circ B)\,\leq\,P(A)P(B)\,,
\end{equation}
where \(A\circ B\) denotes the disjoint occurrence of \(A\) and \(B\).
%\subsection{The case of equality in the BK inequality}
The aim of this note is to understand the equality case in the BK inequality.
%The obstruction to equality rely on
The condition for equality will be expressed with 
the classical notion of witness, that we recall in the next definition.
\begin{definition}[Witness]
  \label{dewi}
  Let \(A\) be a subset of \(\Omega\) and let \(x\) belong to \(A\).
  A subset \(I\) of \(\unn\) is a witness for \(A\) in \(x\) if any \(y\) in \(\Omega\)
  which coincides with \(x\) on the components whose index is in \(I\)
  is automatically in \(A\), i.e.,
  \begin{equation}
    \forall y\in\Omega\qquad
    \bigl[\,\forall i\in I\quad x_i\,=\,y_i\,\bigr]
    \quad\Longrightarrow\quad y\in A\,.
  \end{equation}
\end{definition}
\noindent
%The conditions for equality are most conveniently expressed in 
%the product space 
%%$\smash{\Omega\times\Omega=
%%\{\,0,1\,\}^n\times
%%\{\,0,1\,\}^n}$.
%\begin{equation}
%\Omega\times\Omega\,=\,
%\{\,0,1\,\}^n\times
%\{\,0,1\,\}^n\,.
%\end{equation}
The main obstruction to equality in BK comes from
pairs of configurations for which \(A\) and \(B\) cannot be certified by
disjoint witnesses. 
\begin{definition}[Disjoint witnesses]
  \label{defdisjpi}
  Let \(A\) and \(B\) be two subsets of \(\Omega\), and let 
  $(u,v)$ belong to $A\times B$.
  We say that there exist disjoint witnesses for $A,B$ in $(u,v)$ if 
  there exist two disjoint subsets $I,J$ of $\unn$ such that 
  $I$ is a witness for $A$ in $u$ and 
  $J$ is a witness for $B$ in $v$.
\end{definition}
\noindent
The disjoint occurrence $A\circ B$ of two events $A,B$ 
can be defined 
as the set of the configurations $u$ such that there exist 
disjoint witnesses for $A,B$ in $(u,u)$:
\begin{multline}
    A\circ B
    \,=\,
    \smash{\Big\{}\,u\in A\cap B:\exists\,I,J\subset\unn,\quad
    I\cap J\,=\,\varnothing,
    \\
    I \text{ is a witness for } A \text{ in } u,\quad
    J \text{ is a witness for } B \text{ in } u
  \,\smash{\Big\}}\,.
  \end{multline}
Our first main result is the following characterization of 
the case of equality in the BK
inequality.
\begin{theorem}%[Equality case in the BK inequality]
  \label{thrm1}
  Let \(p\) belong to  \(]0,1[\), and 
  let \(A,B\) be two increasing events included in $\Omega$.
  We have the equality 
  \begin{equation}
    P(A\circ B)\,=\,P(A)P(B)
  \end{equation}
  if and only if 
  every configuration belonging to $A\times B$ admits disjoint witnesses.
  %Any pair $(x,y)$ belonging to $A\times B$ admits disjoint witnesses;
    %For any $(x,y)$ in $A\times B$, there exist two 
    %disjoint subsets $I,J$ of 
    %\(\unn\) such that \(I\) is a witness for \(A\) in \(x\)
    %and \(J\) is a witness o \(B\) in \(y\);
  %There does not exist $(x,y)$ 
    %No configuration in $A\times B$ has a symmetric disjoint pivot.
   % , there exist two 
    %For every \(i\) in \(\unn\),
    %\begin{equation}
      %\Ptwo\bigl(
      %i \text{ is a symmetric disjoint pivot for } (A,B)
      %\bigr)\,=\,0\,.
    %\end{equation}
\end{theorem}
%\FloatBarrier
\noindent
%To prove this theorem, 
%we will study the gap 
  %$\Delta_{BK}(A,B)$
%associated to the BK inequality, 
%defined as
%\begin{equation}
  %\label{defbkintro}
  %\Delta_{BK}(A,B)
  %\,=\,
  %P(A)P(B)-P(A\circ B)\,,
%\end{equation}
%The study of this gap is a difficult and interesting problem in itself, 
%and we will obtain only partial results, which nevertheless 
%allow to prove theorem~\ref{thrm1}.
%The first result in this direction is a slight improvement of the BK inequality,
%that we present in the next subsection.
%\subsection{The strengthened BK inequality}
\FloatBarrier
\noindent
{%\color{blue}
In the next corollary, we show that the equality case for the BK inequality 
is in fact the same as for the FKG inequality.
The case of equality for the FKG inequality 
can be obtained in the uniform case 
with the help of 
Talagrand's inequality \cite{INCT}. In the biased case, it 
follows from the work of Keller and Kindler \cite{KEKI}.
However, these inequalities are far more complicated than what 
is needed to obtain the cases of equality. Understanding the 
cases of equality in correlation inequalities is an important 
goal,
and it is an active research field. Very recently,
Chan and Pak \cite{CHPA} obtained a general result for the 
Ahlswede-Daykin inequality and the FKG inequality in a more general framework.
%The case of equality for the FKG inequality is discussed in the uniform case 
%in the introduction of Talagrand's paper \cite{INCT}. In the biased case, it 
%follows from the work of Keller and Kindler \cite{KEKI}.
\begin{corollary}%[Equality case in the BK inequality]
  \label{coro1}
  %Let \(p\) belong to  \(]0,1[\), and 
  Two increasing events 
  \(A,B\) realize the equality in the BK inequality if and only if they depend 
  on disjoint sets of coordinates.
  %\begin{equation}
    %P(A\circ B)\,=\,P(A)P(B)
  %\end{equation}
  %if and only if 
  %every configuration belonging to $A\times B$ admits disjoint witnesses.
  %Any pair $(x,y)$ belonging to $A\times B$ admits disjoint witnesses;
    %For any $(x,y)$ in $A\times B$, there exist two 
    %disjoint subsets $I,J$ of 
    %\(\unn\) such that \(I\) is a witness for \(A\) in \(x\)
    %and \(J\) is a witness o \(B\) in \(y\);
  %There does not exist $(x,y)$ 
    %No configuration in $A\times B$ has a symmetric disjoint pivot.
   % , there exist two 
    %For every \(i\) in \(\unn\),
    %\begin{equation}
      %\Ptwo\bigl(
      %i \text{ is a symmetric disjoint pivot for } (A,B)
      %\bigr)\,=\,0\,.
    %\end{equation}
\end{corollary}
\begin{proof} As the argument is short, we include it here.
Let 
  \(A,B\) 
  be two increasing events 
  realizing the equality in the BK inequality. 
  Theorem~\ref{thrm1} implies that 
  \begin{equation}
    \label{difg}
  \forall i\in\unn\qquad 
  P(i\text{ is pivotal for } A)
  P(i\text{ is pivotal for } B)\,=\,0\,.
  \end{equation}
Let us set 
\begin{align}
  I\,&=\,\big\{\,i\in\unn:   
P(i\text{ is pivotal for } A)>0\,\big\}\,,\\
    J\,&=\,\big\{\,i\in\unn:   
P(i\text{ is pivotal for } B)>0\,\big\}\,.
\end{align}
It follows from~\eqref{difg} that the sets $I,J$ are disjoint.
Furthermore, the event $A$ (respectively $B$) depends 
on the coordinates $I$ (respectively $J$).
\end{proof}
}
The second main result we would like to present and discuss 
is the strengthened BK inequality, stated in the next theorem.
It involves
the product-space analogue of disjoint occurrence, that we define now.
\begin{definition}[Disjoint testimony]
  \label{distes}
  Let \(A\) and \(B\) be two subsets of \(\Omega\).
  The disjoint testimony of \(A\) and \(B\) is the subset
  \(A\boxtimes B\) of \(A\times B\) consisting of the pairs \((x,y)\)
  admitting disjoint witnesses for \(A\) and \(B\), i.e.,
  \begin{multline}
    A\boxtimes B
    \,=\,
    \smash{\Big\{}\,(x,y)\in A\times B:\exists\,I,J\subset\unn,\quad
    I\cap J\,=\,\varnothing,
    \\
    I \text{ is a witness for } A \text{ in } x,\quad
    J \text{ is a witness for } B \text{ in } y
  \,\smash{\Big\}}\,.
  \end{multline}
\end{definition}
\FloatBarrier
\noindent
\noindent
%The proof of Theorem~\ref{thrm1} is based on a decomposition of the BK gap
%into two non-negative terms induced by a 
We denote by \(\Ptwo\) the product measure on $\Omega\times\Omega$.
\begin{theorem}%[Strong BK inequality]
  \label{thm2}
  For any pair of increasing events \(A,B\) in \(\Omega\), we have
  \begin{equation}
    \label{mainres}
    P(A\circ B)\,\leq\,\Ptwo(A\boxtimes B)\,.
  \end{equation}
\end{theorem}
\noindent
This theorem immediately implies the BK inequality. Indeed, 
the disjoint testimony 
$A\boxtimes B$ is 
a subset of $A\times B$, thus
\begin{equation}
  \label{imbk}
  P(A\circ B)
  \,\leq\,
  \Ptwo(A\boxtimes B)
  \,\leq\,
  \Ptwo(A\times B)
  \,=\,
  P(A)P(B)\,.
\end{equation}
%Theorem~\ref{thm2} readily implies the BK inequality.
%Indeed, 
%for any pair of increasing events $A,B$, the disjoint 
%testimony
%$\Abox{A}{B}$ is a subset of the product $A\times B$, whence 
%\begin{equation}
    %\label{jmbk}
%\Ptwo(A\boxtimes B)
%\,\le\, 
%\Ptwo(A\times B)\,=\,P(A)P(B)
%\,.
%\end{equation}
In their recent paper \cite{RT},
Radhakrishnan and Tassion say that this inequality is known
(see \cite{RT},
formula~(13) in remark~$7$).
They say that it follows easily from 
equation~$2.6$ in~\cite{BE}.
This equation states the following result:
%\begin{adjustwidth}{-1.5em}{-1.5em}
\begin{quote}%[Equation $2.6$ of~\cite{BE}]
  \itshape
  "In \cite{BK} also the following stronger result has been shown:
\begin{equation}
  \label{aa}
\mu\Big(
\bigcup_{1 \leq i \leq k} A_i \square B_i
\Big)
\leq
(\mu \times \mu)\Big(
\bigcup_{1 \leq i \leq k} A_i \times B_i
\Big),
\tag{2.6}
\end{equation}
where $A_i,B_i$ are increasing subsets of $\Omega$,
$i = 1,\dots,k$."
\end{quote}
%\end{adjustwidth}
\smallskip 

\noindent
In the above statement, the measure~$\mu$ is a probability measure 
on~$\Omega$ which belongs to a class called NBU, but we do not 
need to worry about these details, because the Bernoulli product
measure belongs to this class. The article \cite{BK}, referred to
in the statement, is the original paper of Van den Berg and Kesten.
While it is indeed true that the inequality~\eqref{mainres}
follows from~\eqref{aa}, it is not completely straightforward
and it requires a bit of thought. In any case, apart 
from the 
remark~$7$
in \cite{RT}, to the best of our knowledge, we are not aware 
of any other place where the inequality~\eqref{mainres} 
could be found in the literature.
{%\color{blue}
If we focus on the Bernoulli product measure, as we do here, 
one would like to have a proof simpler than the one of 
the original paper \cite{BK}, which rests on the notion of NBU 
measure.
The paper 
\cite{RT} obtains the 
inequality~\eqref{mainres} as a byproduct of a more complicated 
argument, but only in the symmetric case $p=1/2$.
Nowadays, the standard way to prove the BK inequality is the 
splitting method. 
The splitting method was used by van den Berg and
Fiebig to prove a more general version of the inequality, where 
$A$ and $B$ are intersections of increasing events 
(see theorem~$4.2$ of \cite{BF}).
It is closely related to the proof of the clutter theorem by
McDiarmid \cite{DI1,DI2} (see also 
\cite{BE}, page 358, for a further presentation).
%In the classical reference 
%and it is also the one presented in 
Grimmett's book \cite{GR}, the standard reference 
on percolation, presents a proof of the BK inequality 
for increasing events based on the splitting method.
It turns out that this proof yields also
theorem~\ref{thm2}. 
}
%can be obtained with the same proofproved with a minor adapalong the same lines.
%The proof of is a slight modification of 
%proof of the BK inequality 
%presented in Grimmett's book \cite{GR}. 
In fact, it is essentially a matter of 
giving the adequate definitions.
Let us precise a bit this point, in order to lift 
some ambiguities which gave us some hard time.
When working on the
  product $\zuen\times\zuen$, one can think naturally 
  of defining the disjoint occurrence of two events 
  $A,B$ included in $\zuen$ in two different ways:
  \smallskip

  \noindent
  $\bullet$ we can consider 
  $\zuen\times\zuen$ as $\zu^{2n}$, and we use 
  the definition of the disjoint occurrence $\circ$ on
  the $2n$-th product of $\zu$; 
  \smallskip

\noindent 
  $\bullet$ or we can consider 
  $\zuen\times\zuen$ as $\smash{\big(\zu^{2}\big)^n}$, and we use 
  the definition of the disjoint occurrence $\circ$ on 
  the $n$-th product of $\zu^2$. 
  \smallskip

  \noindent
It seems that the first choice is used for instance 
at the beginning of the proof
in Grimmett's book \cite{GR}. This is indeed what is suggested 
page~40 after formula~$2.20$
by the identity $A'\circ B'_m=A'\cap B'_m$. 
At the same time, Grimmett's proof rests in a crucial 
way on the decomposition of the events involving the
disjoint testimony, as in definition~\ref{distes}, and 
this 
corresponds to the second choice above.
We believe that the most transparent presentation would 
consist 
in working all the time with the disjoint testimony.
%This is why we introduce the events
%$\Aboxk{A}{k}{B}$ instead of $A\circ_k B$, in order 
%to avoid 
%the aforementioned ambiguity.
With this in mind, 
%it is clear that 
%This being said, 
%the main contribution of this text is to 
%rewrite 
Grimmett's proof 
%in light of these definitions, 
%and it readily 
yields also the inequality~\eqref{mainres}
(although this inequality is not formally stated in \cite{GR}). 
%it is a byproduct of the proof of the BK inequality presented 
%there.
%some specific subsets of the events 
%$A'\circ B'_k$ which corresponds to the definition 
%the proof uses
%It is also the point 
%of view adopted in the lecures notes of Duminil-Copin \cite{DC},
%see the exercice~12 for the scheme of proof of the BK inequality.

%It seems to us that the modification makes the proof more 
%transparent, in addition to providing a stronger inequality!

%\noindent
In any case, 
we will propose in 
section~\ref{proth2}
a new proof of the strengthened 
BK inequality. 
The general strategy is the same as Grimmett's proof,
that is we perform an induction over the coordinates
and at each step we prove an adequate inequality which 
allows to replace a diagonal coordinate by a pair 
of independent ones.
%compare the probability of an event with 
%its counterpart in a product space.
%by replacing a diagonal coordinate by an independent one is exactly detected by
%symmetric disjoint pivots.
% which rests on the notion of symmetric disjoint pivots and is inspired by 
The key argument to understand the case of equality in BK rests 
on the notion of 
symmetric disjoint pivots
(defined in subsection~\ref{subspdf}), 
already present 
%However, the core of the argument 
%is inspired by the techniques used by 
%Radhakrishnan and Tassion \cite{RT}.
%In fact, the symmetric disjoint pivots play a central role 
%In fact, this notion  plays also a central role 
%in the arguments of their work \cite{RT}: 
in the work 
of Radhakrishnan and Tassion \cite{RT}. 
%the disjoint occurrence $A\circ B$.
%The advantage of this proof is that the core argument 
%it is more symmetric, and it 
Instead of proving an inequality during the induction step,
we will prove an equality which 
puts forward the role played by the 
symmetric disjoint pivots. This equality is the key 
to characterize the events realizing the equality in BK.
\section{Proofs}
\label{proth2}

This section is devoted to the proofs of 
the two theorems~\ref{thrm1} and~\ref{thm2}.
We prove first 
the strengthened BK inequality stated in theorem~\ref{thm2}.
A careful examination of the proof will yield the 
case of equality in BK stated in theorem~\ref{thrm1}.
Rob van den Berg kindly pointed out to us that, in fact, 
%once the condition for equality is known, 
%the corollary~\ref{coro1} could be 
theorem~\ref{thrm1} 
%the equality case of the BK inequality can be 
can also be obtained quite quickly from the very statement of theorem~\ref{thm2}. 
%However, the proof of theorem~\ref{thm2} presented here 
%highlights the role of the 
%symmetric disjoint pivots

%Let us introduce some additional notation.
\subsection{Hybrid configurations and testimony}
%that we use in the proof.
The proof of the strengthened BK inequality relies on 
an interpolation scheme. 
In this subsection, we introduce some additional notation
necessary to describe this scheme.
For \(x,y\) in \(\Omega\) and \(k\) in \(\zun\),
we define the hybrid
configuration \(x*_{k}y\) obtained by taking the first \(k\) coordinates of \(y\) and the last \(n-k\) coordinates of \(x\), namely
\[
x*_{k}y \,=\, (y_1,\dots,y_k,x_{k+1},\dots,x_n)\,.
\]
For instance, we have
\begin{equation}
  \label{exon}
x*_{0}y \,=\, x\,,
\qquad
x*_{n}y \,=\, y\,.
\end{equation}
The disjoint testimony of two
    subsets
   \(A\) and \(B\)
    of \(\zuen\) was introduced in the definition~\ref{distes}.
In the proof, we will use the following equivalent definition:
\begin{multline}
A \boxtimes B\,
=\,
\smash{\Bigl\{}\,(x,y)\in \Omega^2 : \exists\, I\subset \unn,\
I \text{ witness for } A \text{ in } x,\\
I^c \text{ witness for } B \text{ in } y\,\smash{\Bigr\}}\,.
\end{multline}
%\begin{proof}[Proof of Theorem \ref{thm2}]
For \(k\) in \(\zun\), we introduce an intermediate version of the disjoint testimony,
in which the witness for \(B\) is tested not on \(y\) itself but on the hybrid
configuration \(x *_k y\).
More precisely, we define
\begin{multline}
  \forall k\in\zun\qquad\\
\Aboxk{A}{k}{B}
\,=\,
\smash{\Bigl\{}\,(x,y)\in \Omega^2 : \exists\, I\subset \unn,\
I \text{ witness for } A \text{ in } x,\,\\
I^c \text{ witness for } B \text{ in } x*_{k}y
\, \smash{\Bigr\}}\,.
\end{multline}
In particular, it follows from~\eqref{exon} that
\begin{equation}
  \label{cif}
\Aboxk{A}{0}{B}\,=\,(\Aocc{A}{B})\times \Omega\,,
\qquad
\Aboxk{A}{n}{B}\,=\,A \boxtimes B\,.
\end{equation}
As \(k\) increases from \(0\) to \(n\), the first \(k\)
coordinates are progressively transferred from \(x\) to \(y\),
and the events
\(\Aboxk{A}{k}{B}\), \(0\leq k\leq n\),
form a sequence of events 
interpolating between the disjoint occurrence on a single configuration and the disjoint
testimony on two independent configurations.
We will prove that
  \begin{equation}
    \label{recin}
  \forall k\in\unn\qquad
\Ptwo(\Aboxk{A}{k-1}{B})\,\leq\, \Ptwo(\Aboxk{A}{k}{B})\,.
  \end{equation}

\noindent
In view of~\eqref{cif}, this will readily imply the desired
inequality~\eqref{mainres}.
%The proof of inequality~\eqref{recin} relies on a local identity, which is the key to the argument.
%We define
%\[
    %\phi(u,v)
    %\,=\,
     %1_{\{\,(u,v)\in A\boxtimes B\,\}}\,.
%\]
%For a coordinate \(k\), and for $a,b$ in $\{\,0,1\,\}$, 
%we define also 
%\[
    %\phi_{ab}^k(u,v)
    %\,=\,
    %\phi\big(\setcoord{u}{k}{a},\setcoord{v}{k}{b}\big)\,.
%\]
%\noindent
%The next lemma is the key of the interpolation argument: it shows that 
%the gain obtained
%by replacing a diagonal coordinate by an independent one is exactly detected by
%symmetric disjoint pivots.
%We first rewrite the definition of a symmetric disjoint pivot
%with the notation
%$A\boxtimes B$ 
%for the disjoint testimony.
%%The second obstruction corresponds to the following situation: disjoint testimony
%%is impossible when the coordinate \(i\) is closed in both configurations, but it
%%becomes possible if this coordinate is opened in either one of the two
%%configurations.
  %Let \(A\) and \(B\) be two subsets of \(\Omega\), and let \(i\) in \(\unn\).
  %We say that \(i\) is a symmetric disjoint pivot for \((A,B)\) in 
  %the pair of configurations \((u,v)\)
  %if
  %\begin{gather}
    %\bigl(\setcoord{u}{i}{0},\setcoord{v}{i}{0}\bigr)
    %\notin A\boxtimes B\,,\\
  %%whereas
    %%\end{equation}
    %%\begin{equation}
    %\bigl(\setcoord{u}{i}{0},\setcoord{v}{i}{1}\bigr)
    %\in A\boxtimes B\,,\quad
    %\bigl(\setcoord{u}{i}{1},\setcoord{v}{i}{0}\bigr)
    %\in A\boxtimes B\,.
  %\end{gather}
%\noindent
%\subsection{Diagonal versus independent coordinate}
%\color{blue}
\subsection{The splitting method}
We discuss here the core argument 
of the splitting method 
used by van den Berg and
Fiebig \cite{BF}.
%The goal at each step is 
The idea consists in 
replacing 
at each step  
a diagonal coordinate by an independent one. 
As there are only two random variables involved,
it is more convenient to do the computation for 
a boolean function 
$\phi:\zu^2\to\zu$ depending on two coordinates.
The main question can be restated as follows.
Let $X,Y$ be two independent Bernoulli variables with 
parameter~$p$. When do we have 
\begin{equation}
  E\big(\phi(X,X)\big)\,\leq\,
  E\big(\phi(X,Y)\big)\,?
\end{equation}
Let us compute the difference of the expectations:
\begin{multline}
  \label{compdiff}
  E\big(\phi(X,Y)\big)\,-
  E\big(\phi(X,X)\big)\,=\,
\Bigl(
(1-p)^2\phi({0,0})
+p(1-p)\phi({1,0})
\\+p(1-p)\phi({0,1})
+p^2\phi({11})
\Bigr)
%\\
-
\Bigl(
(1-p)\phi({0,0})
+p\phi({1,1})
\Bigr)\\
\,=\,
p(1-p)
\big(
    \phi({1,0}) + \phi({0,1})
    - \phi({0,0}) - \phi({1,1})
\big)\,.
\end{multline}
For a general boolean function $\phi$, the quantity
\begin{equation}
  \label{posval}
    \Delta\phi\,=\,
    \phi({0,0}) + \phi({1,1})
    - 
    \phi({1,0}) - \phi({0,1})
\end{equation}
can take the five values $-2,-1,0,1,2$.
We will be concerned only with the case where the 
function $\phi$ is non-decreasing. In this case, the only 
%possible values for~\eqref{posval} are $-1,0,1$.
possible values for $\Delta\phi$ are $-1,0,1$.
If $\phi$ is constant, it is equal to $0$. Apart from the constant 
boolean functions, there are four non-decreasing boolean functions 
in two variables, which are:
  \smallskip

\noindent 
  $\bullet$ The double pivot: 
  $\phi(1,1)=1$, $\phi(1,0)= \phi(0,1)= \phi(0,0)= 0$. 
  \smallskip

\noindent 
  $\bullet$ The left pivot: 
  $\phi(1,1)=\phi(1,0)=1$, $ \phi(0,1)= \phi(0,0)= 0$. 
  \smallskip

\noindent 
  $\bullet$ The right pivot: 
  $\phi(1,1)=\phi(0,1)=1$, $ \phi(1,0)= \phi(0,0)= 0$. 
  \smallskip

\noindent 
  $\bullet$ The symmetric pivot: 
  $\phi(1,1)=\phi(1,0)= \phi(0,1)=1$, $\phi(0,0)= 0$. 
  \smallskip

\noindent 
A direct check shows that
\begin{equation}
  \label{dirche}
    \Delta\phi\,=\,
    \begin{cases}
      \phantom{-}1 &\text{ if }\phi\text{ is the double pivot}\,,\cr
      -1 &\text{ if }\phi\text{ is the symmetric pivot}\,,\cr
      \phantom{-}0 &\text{ otherwise}\,.
    \end{cases}
\end{equation}
We sum up the previous observations in the next lemma.
\begin{lemma}
  \label{corelem}
  Let $\phi$ be a non-decreasing boolean function depending on 
  two coordinates and 
let $X,Y$ be two independent Bernoulli variables with 
parameter~$p$. We have
 \begin{multline}
  \label{observ}
  E\big(\phi(X,Y)\big)\,-
  E\big(\phi(X,X)\big)\,=\,\cr
  p(1-p)\Big(
    1\big({\phi\text{ is the symmetric pivot}}\big)-
    1\big({\phi\text{ is the double pivot}}\big)
    \Big)\,.
\end{multline} 
\end{lemma}
\noindent
We conclude naturally that 
  \begin{equation}
  E\big(\phi(X,X)\big)\,\leq\,
  E\big(\phi(X,Y)\big)
  \end{equation}
if and only if $\phi$ is not the double pivot.
\subsection{The symmetric disjoint pivots}
\label{subspdf}
One source of 
obstruction to equality in BK
is due to the existence of 
%can be expressed with the notion 
symmetric disjoint pivots,
a new notion 
that we 
introduce in the next definition. 
%\noindent
%{\bf Coordinate replacement.} Let $x=(x_1,\cdots,x_n)$ be a 
%configuration in $\Omega$. For $i$ in $\unn$, we denote by 
%$x(i\leftarrow 0)$
%(respectively $x(i\leftarrow 1)$) 
%the configuration $x$ where the $i$-th component is set
%to $0$ (respectively to $1$), i.e.,
%\begin{align}
  %x(i\leftarrow 0)\,&=\,(x_1,\dots,x_{i-1},0,x_i,\dots,x_n)\,,\cr
  %x(i\leftarrow 1)\,&=\,(x_1,\dots,x_{i-1},1,x_i,\dots,x_n)\,.
%\end{align}
%Given a configuration 
For $u =(u_1,\cdots,u_n)$ a configuration in $\Omega$ and $i$ an
index in 
$\unn$, 
we denote by $u(i\leftarrow 0)$
(respectively $u(i\leftarrow 1)$) 
the configuration $u$ where the $i$-th component is set
to $0$ (respectively to $1$), i.e.,
\begin{align}
  u(i\leftarrow 0)\,&=\,(u_1,\dots,u_{i-1},0,u_i,\dots,u_n)\,,\cr
  u(i\leftarrow 1)\,&=\,(u_1,\dots,u_{i-1},1,u_i,\dots,u_n)\,.
\end{align}
\begin{definition}[Symmetric disjoint pivot]
  \label{defsdpi}
  Let \(A,B\) be two subsets of~\(\Omega\).
  Let $(u,v)$ be an element of $\Omega\times \Omega$, and let \(i\) belong to 
 \(\unn\).
  We say that \(i\) is a symmetric disjoint pivot for \((A,B)\) in \((u,v)\)
  if there are disjoint witnesses for the events $A,B$ in
    $\bigl(\setcoord{u}{i}{1},\setcoord{v}{i}{0}\bigr)$ 
    and 
    $\bigl(\setcoord{u}{i}{0},\setcoord{v}{i}{1}\bigr)$,
    but not in 
    $\bigl(\setcoord{u}{i}{0},\setcoord{v}{i}{0}\bigr)$.
\end{definition}
%equipped with the product measure $P\times P$.
\noindent
\begin{figure}[hbt]
\vbox{
  \centerline{A configuration $u$ realizing the event $\{\,A\longleftrightarrow B\,\}$}
  \bigskip
\centerline{
% ---- SEULE LIGNE A CHANGER POUR REDIMENSIONNER ----
\psset{unit=1.084cm}
% -----------------------------------------------------
%\pspicture(-11,-11)(11,11)
\pspicture(-5,-3)(5,3)
\psset{dotsize=0.24}

% Grille de points de fond, en coordonnées pures (pas de \newdimen)
\multido{\iX=-5+1}{11}{
  \multido{\iY=-3+1}{7}{
    \psdots[fillcolor=black,dotstyle=x](\iX,\iY)
  }
}

\psset{dotsize=0.27}

%\psdots[fillcolor=white,dotstyle=o](-5,-2)(-4,-2)(-4,-1)(-4,0)(-3,0)(-2,0)
%\psdots[fillcolor=black,dotstyle=o](-5,2)(-4,2)(-4,1)(-3,1)(-2,1)
\psdots[fillcolor=black,dotstyle=o](-4,1)(-3,1)(-2,1)
\psdots[fillcolor=black,dotstyle=o](-2,0)(2,1)(2,0)
%\psdots[fillcolor=black,dotstyle=o](5,-2)(4,-2)(4,-1)(4,0)(3,0)(2,0)
%\psdots[fillcolor=black,dotstyle=o](5,2)(4,2)(4,1)(3,1)(2,1)
\psdots[fillcolor=black,dotstyle=o](4,1)(3,1)(2,1)
%\psdots[fillcolor=black,dotstyle=o](-4,3)(4,3)
\psdots[fillcolor=gray,dotstyle=square](-2,2)(-1,2)(0,2)(1,2)(2,2)
\psdots[fillcolor=gray,dotstyle=square](-2,-2)(-1,-2)(0,-2)(1,-2)(2,-2)(2,-1)(-2,-1)

%\rput(-4.35,1){$a$}
%\rput(4.35,1){$b$}
%\rput(-4.35,-1){$c$}
%\rput(4.35,-1){$d$}

\rput(-4.35,1){$A$}
\rput(4.35,1){$B$}
\rput(-4.35,-1){$C$}
\rput(4.35,-1){$D$}
%

%\psdot[fillcolor=white,dotstyle=x](-3,4)
%\rput(-2,4){closed site}
%\psdot[fillcolor=black,dotstyle=o](2,6)
%\psdot[dotstyle=o](2,4)
%\rput(3,4){open site}

\endpspicture
}
\bigskip
\medskip
\bigskip
  \centerline{A configuration $v$ realizing the event $\{\,C\longleftrightarrow D\,\}$}
  \bigskip
\centerline{
% ---- SEULE LIGNE A CHANGER POUR REDIMENSIONNER ----
\psset{unit=1.1cm}
% -----------------------------------------------------
%\pspicture(-11,-11)(11,11)
\pspicture(-5,-3)(5,3)
\psset{dotsize=0.24}

% Grille de points de fond, en coordonnées pures (pas de \newdimen)
\multido{\iX=-5+1}{11}{
  \multido{\iY=-3+1}{7}{
    \psdots[fillcolor=black,dotstyle=x](\iX,\iY)
  }
}

\psset{dotsize=0.27}

%\psdots[fillcolor=white,dotstyle=o](-5,2)(-4,2)(-4,1)(-4,0)(-3,0)(-2,0)
%\psdots[fillcolor=black,dotstyle=o](-5,-2)(-4,-2)(-4,-1)(-3,-1)(-2,-1)
\psdots[fillcolor=black,dotstyle=o](-5,-1)(-4,-1)(-3,-1)
\psdots[fillcolor=black,dotstyle=o](-5,-1)(-5,0)(-5,1)(-5,2)(-4,2)(-3,2)
%\psdots[fillcolor=black,dotstyle=o](5,-2)(4,-2)(4,-1)(3,-1)(2,-1)
\psdots[fillcolor=black,dotstyle=o](5,-1)(4,-1)(3,-1)
\psdots[fillcolor=black,dotstyle=o](5,-1)(5,0)(5,1)(5,2)(4,2)(3,2)
%\psdots[fillcolor=white,dotstyle=o](5,2)(4,2)(4,1)(4,0)(3,0)(2,0)
%\psdots[fillcolor=black,dotstyle=o](-4,3)(4,3)
\psdots[fillcolor=gray,dotstyle=square](-2,2)(-1,2)(0,2)(1,2)(2,2)
\psdots[fillcolor=gray,dotstyle=square](-2,-2)(-1,-2)(0,-2)(1,-2)(2,-2)(2,-1)(-2,-1)

\rput(-4.35,1){$A$}
\rput(4.35,1){$B$}
\rput(-4.35,-1){$C$}
\rput(4.35,-1){$D$}
%
%\rput(-4.35,1){$a$}
%\rput(4.35,1){$b$}
%\rput(-4.35,-1){$c$}
%\rput(4.35,-1){$d$}

%\rput(-5.5,2){$a$}
%\rput(5.5,2){$b$}
%\rput(-5.5,-2){$c$}
%\rput(5.5,-2){$d$}

%\psdot[fillcolor=white,dotstyle=x](-3,4)
%\rput(-2,4){closed site}
%\psdot[fillcolor=black,dotstyle=o](2,6)
%\psdot[dotstyle=o](2,4)
%\rput(3,4){open site}

\endpspicture
}
\bigskip
\medskip
\bigskip
\centerline{
  Two configurations of site percolation in the rectangle 
  $[-5,5]\times[-3,3]$
}
\medskip
\centerline{
\psset{dotsize=0.25}
 %\psdot[dotstyle=x](0,0) 
 \raisebox{0.6ex}{\psdot[dotstyle=x](0,0)}
 \kern5pt=\kern3pt closed site\,,\qquad
 \raisebox{0.6ex}{\psdot[fillcolor=black,dotstyle=o](0,0)}
 \kern5pt
 and 
 \kern5pt
 \raisebox{0.6ex}{\psdot[fillcolor=gray,dotstyle=square](0,0)}
 \kern5pt=\kern3pt open site\,,\qquad
 \raisebox{0.6ex}{\psdot[fillcolor=gray,dotstyle=square](0,0)}
 \kern5pt=\kern3pt 
symmetric disjoint pivot
}
\bigskip
\caption{The pair $(u,v)$ is in 
  $\{\,A\longleftrightarrow B\,\}\boxtimes
  \{\,C\longleftrightarrow D\,\}$
%Symmetric pivots in percolation
}
\label{figpat}
}
\end{figure}\noindent
\noindent
In fact, this notion  plays also a central role 
%in the arguments of their work \cite{RT}: 
in the work 
of Radhakrishnan and Tassion \cite{RT}: 
when $F$ is taken to 
be the indicator function of 
$A\circ B$,
the function 
$D_iF$ defined in equation~$(6)$ of \cite{RT} is the indicator 
function of a symmetric disjoint pivot. 

The BK inequality was invented by Van den Berg and Kesten \cite{BK} in 
relationship with a percolation problem. Figure~\ref{figpat}
presents two configurations of site percolation realizing two 
connection events, and the symmetric disjoint pivots are 
marked by gray squares.
%\FloatBarrier

Surprisingly, 
forcing the existence of disjoint witnesses for all the 
configurations
%the condition on the 
%disjoint witnesses of theorem~\ref{thrm1} 
implies the absence of 
symmetric disjoint pivots, as stated in the next lemma.
\begin{lemma} 
  \label{suprilem}
  Let $A,B$ be two increasing events. 
If every configuration belonging to $A\times B$ admits disjoint witnesses,
then 
    no configuration in $\Omega\times \Omega$ has a symmetric disjoint pivot.
\end{lemma}
\begin{proof}
  Suppose that 
  every configuration belonging to $A\times B$ admits disjoint witnesses.
Then the disjoint testimony 
  \(A\boxtimes B\) coincides with the event 
  %\(A\boxtimes B\) is equal to the event 
  $A\times B$.
  Let $(u,v)$ be an element of 
  $\Omega\times \Omega$. 
  Let $i$ be an index in $\unn$ and suppose that 
  \begin{gather}
    \label{sura}
    \bigl(\setcoord{u}{i}{1},\setcoord{v}{i}{0}\bigr) \,\in\,
  A\boxtimes B\,,\\
    \label{surb}
    \bigl(\setcoord{u}{i}{0},\setcoord{v}{i}{1}\bigr) \,\in\,
  A\boxtimes B\,.
  \end{gather}
As 
  $A\boxtimes B$ is equal to 
  $A\times B$, we deduce 
  from~\eqref{sura} that $\setcoord{v}{i}{0}$ is in $B$,   
  and 
  from~\eqref{surb} that $\setcoord{u}{i}{0}$ is in $A$.   
  Therefore 
  \begin{equation}
    \bigl(\setcoord{u}{i}{0},\setcoord{v}{i}{0}\bigr) \,\in\,
  A\times B\,,
  \end{equation}
  but using again that 
  $A\times B=A\boxtimes B$, we see that  
  $\smash{\bigl(\setcoord{u}{i}{0},\setcoord{v}{i}{0}\bigr)}$ realizes also the disjoint 
    testimony of $A$ and $B$, and we conclude that $i$ is not a 
    symmetric disjoint pivot!
\end{proof}
\FloatBarrier
\subsection{Proof of the inequality at step~$k$}
We proceed now to the proof of the inequality~\eqref{recin}.
Let
us fix \(k\) in \(\{\,1,\dots,n\,\}\).
At step \(k-1\), the coordinate \(k\) of the configuration
\(x*_{k-1}y\) is equal to \(x_k\), whereas at step \(k\), it is equal to \(y_k\). 
%Once the other coordinates have been fixed, 
%in order to compare 
Thus the
transition from \(\Aboxk{A}{k-1}{B}\) to
\(\Aboxk{A}{k}{B}\) consists in 
replacing the diagonal pair \((x_k,x_k)\) by
the independent pair \((x_k,y_k)\). 
To compare the probabilities of the events
\(\Aboxk{A}{k-1}{B}\) and
\(\Aboxk{A}{k}{B}\), we will
perform a conditioning on all the coordinates other than $k$,
%$\smash{\big((x_i,y_i),1\leq i\leq n, {i\neq k}\big) }$,
%\((x_i,y_i)_{i\neq k}\), 
and we will make appeal to lemma~\ref{corelem}.
So we fix a choice
%we perform a conditioning on all 
%the configuration of 
for the coordinates other than $k$:
\begin{equation}
\big((x_i,y_i),1\leq i\leq n, {i\neq k}\big) \,.
\end{equation}
%of the coordinates apart from $k$. 
%coordinates
%\((x_i,y_i)_{i\neq k}\). 
%\FloatBarrier\noindent 
We define the partial  map $\phi_k:\zu^2\to\zu$ by setting
\begin{multline}
  \label{defphik}
  \forall a,b\in\zu\qquad
\phi_k(a,b) \,=\,\cr
     1_{A\boxtimes B}\big(
       (x_1,\dots,x_{k-1},a,x_{k+1},\dots,x_n),
       (y_1,\dots,y_{k-1},b,x_{k+1},\dots,x_n)
\big)\,.
\end{multline}
With these definitions, we have, for any 
  $x_k,y_k$ in $\zu$,
  %\forall x_k,y_k\in\zu\qquad
\begin{align}
  %\forall x_k,y_k\in\zu\qquad
  \label{jdfa}
  \phi_k(x_k,x_k) &\,=\,
     1_{A\boxtimes B}(x, x*_{k-1}y)\,=\,
     1_{\Aboxk{A}{k-1}{B}}(x,y)\,,\\
  \label{jdfb}
\phi_k(x_k,y_k) &\,=\,
     1_{A\boxtimes B}(x, x*_{k}y)\,=\,
     1_{\Aboxk{A}{k}{B}}(x,y)\,.
\end{align}
We compute next the 
conditional probability of 
\(\Aboxk{A}{k-1}{B}\) 
with the help of~\eqref{jdfa}.
Denoting by $E_k$ the conditional expectation with respect to 
the $k$-th coordinate once the others are fixed and equal to 
$(x_i,y_i),1\leq i\leq n, {i\neq k}$, we have
\begin{equation}
P\Big( \Aboxk{A}{k-1}{B} \,\big|\, 
(x_i,y_i),1\leq i\leq n, {i\neq k}\Big) 
%\cr
\,=\,
E_k\big(
     1_{\Aboxk{A}{k-1}{B}} \big)
\,=\,
E_k\big( \phi_k(x_k,x_k)\big)\,.
\end{equation}
Similarly, using~\eqref{jdfb}, we have
\begin{equation}
P\Big( \Aboxk{A}{k}{B} \,\big|\, (x_i,y_i),1\leq i\leq n, {i\neq k}\Big) 
\,=\,
E_k\big( \phi_k(x_k,y_k)\big)\,.
\end{equation}
Let us denote by
$P_k$ the conditional probability  
knowing that the coordinates other than $k$ are equal to
$(x_i,y_i),1\leq i\leq n, {i\neq k}$.
%The conditional distribution of $(x_k,y_k)$ knowing all the other coordinates is simply a 
The probability $P_k$ is simply the
     product of two Bernoulli distributions of parameter $p$. 
Furthermore, the map $\phi_k$ is non-decreasing, so we are in position 
to apply lemma~\ref{corelem}, and we get
\begin{multline}
  \label{posapp}
%P\Big( \Aboxk{A}{k}{B} \,\big|\, (x_i,y_i),1\leq i\leq n, {i\neq k}\Big) 
%-
%P\Big( \Aboxk{A}{k-1}{B} \,\big|\, (x_i,y_i),1\leq i\leq n, {i\neq k}\Big) 
P_k\big( \Aboxk{A}{k}{B} \big)
-
P_k\big( \Aboxk{A}{k-1}{B} \big)
%\cr
\,=\,
E_k\big( \phi_k(x_k,y_k)\big)-
E_k\big( \phi_k(x_k,x_k)\big)
\,=\,\qquad
\cr
p(1-p)
\Big(
1 \big(\phi_k\text{ is a symmetric pivot} \big)
-1 \big(\phi_k\text{ is a double pivot} \big)
\Big)
  \,.
\end{multline}
The key point is stated in the following lemma.
It is due to the very specific structure of the disjoint testimony event.
\begin{lemma}
  \label{keypoi}
For any choice of 
$\big((x_i,y_i),1\leq i\leq n, {i\neq k}\big)$,
%\begin{equation}
%\big((x_i,y_i),1\leq i\leq n, {i\neq k}\big) \,,
%\end{equation}
the 
%boolean 
function $\phi_k$ defined in~\eqref{defphik}
is not a double pivot.
\end{lemma}
\begin{proof}
Let us fix
$\big((x_i,y_i),1\leq i\leq n, {i\neq k}\big)$
%and suppose that the function $\phi_k$ is a double pivot.
%We would then have $\phi_k(1,1)=1$, meaning that
and suppose that 
$\phi_k(1,1)=1$. This means that
\begin{equation}
       \Big(\big(x_1,\dots,x_{k-1},1,x_{k+1},\dots,x_n\big),
       \big(y_1,\dots,y_{k-1},1,x_{k+1},\dots,x_n\big)\Big)\in
     {A\boxtimes B}\,.
\end{equation}
By the very definition of 
     ${A\boxtimes B}$, there exists a subset $I$ of $\unn$ such that
\begin{gather}
I \text{ is a witness for } A \text{ in } 
      \big(x_1,\dots,x_{k-1},1,x_{k+1},\dots,x_n\big)\,,\\
I^c \text{ is a witness for } B \text{ in } 
       \big(y_1,\dots,y_{k-1},1,x_{k+1},\dots,x_n\big)\,.
\end{gather}
%The point is now that $k$ cannot be simultaneously in $I$ and $I^c$.
%So we have two cases:
We consider two subcases:
\smallskip

\noindent
$\bullet$
If $k$ is in $I$, then \(I^c\) is still a witness for \(B\) in 
$\smash{\big(y_1,\dots,y_{k-1},0,x_{k+1},\dots,x_n\big)}$, hence 
$\phi_k(1,0)=1$;
\smallskip

\noindent
$\bullet$
If $k$ is not in $I$, then \(I\) is still a witness for \(A\) in 
$\smash{\big(x_1,\dots,x_{k-1},0,x_{k+1},\dots,x_n\big)}$, hence
$\phi_k(0,1)=1$.
\smallskip

\noindent
Therefore,
either
$\phi_k(1,0)=1$ or
%$\phi_k(0,1)=1$. Thus $\phi_k$ is not a double pivot.
$\phi_k(0,1)=1$, and $\phi_k$ is not a double pivot.
\end{proof}
\noindent
Applying lemma~\ref{keypoi}, we obtain that
\begin{equation}
  \label{applelm}
1 \big(\phi_k\text{ is a double pivot} \big)\,=\,0\,.
\end{equation}
Substituting~\eqref{applelm} into~\eqref{posapp}, we conclude that
\begin{equation}
  \label{sibh}
P_k\big( \Aboxk{A}{k}{B} \big)
-
P_k\big( \Aboxk{A}{k-1}{B} \big)
=
p(1-p)
1\big(\phi_k\text{ is a symmetric pivot} \big)
  .
\end{equation}
We observe next that 
$\phi_k$ is a symmetric pivot if and only if 
$k$ is a symmetric disjoint
pivot for $(A,B)$  in $(x,x*_k y)$.
Taking expectations in~\eqref{sibh}, we get finally
\begin{multline}
  \label{secgap}
\Ptwo(\Aboxk{A}{k}{B})
 - \Ptwo(\Aboxk{A}{k-1}{B}) 
\,=\, 
 \\
p(1-p)\,
\Ptwo
  \Bigg(
  \Bigg\{
\begin{matrix}
    (x,y)\in\Omega\times\Omega:
k \text{ is a symmetric disjoint}\\
\text{pivot for } (A,B) \text{ in } (x,x*_k y)
\end{matrix}
  \Bigg\}
  \Bigg)\,.
%\Ptwo\Big(k \text{ is a symmetric disjoint pivot for } (A,B) \text{ in } (x,x*_k y)\Big)\,.
\end{multline}
This readily yields 
the inequality~\eqref{recin}, and this 
concludes the proof of the 
strengthened BK inequality~\eqref{mainres}.
%\section{Consequences for the BK gap}
%\section{The BK gap}
%\newpage
\subsection{The case of equality}
We will complete here the proof of theorem~\ref{thrm1}.
%We complete now the proofs of the results stated in the introduction. 
  Let $A,B$ be two increasing events.
It follows from the proof of theorem~\ref{thm2} that 
\begin{multline}
  \label{chainine}
P(A\circ B)\,=\,
  \Ptwo(\Aboxk{A}{0}{B}) 
\,\leq\,
  \Ptwo(\Aboxk{A}{1}{B}) 
\,\leq\,\cdots
\cr
\,\leq\,
\Ptwo(\Aboxk{A}{n}{B})
\,=\,
    \Ptwo(A\boxtimes B)
\,\leq\,
P(A)P(B)\,.
\end{multline}
Suppose furthermore that the events $A,B$ realize the equality
$$P(A\circ B)\,=\,P(A)P(B)\,.$$
Then the chain of inequalities~\eqref{chainine} consists
only of equalities. In particular, if we look at the 
last inequality, we must have 
\begin{equation}
  \label{refvv}
    \Ptwo(A\boxtimes B) \,=\, P(A)P(B)
\,=\,
  \Ptwo(A\times B)
%\,,\\
  %\label{refww}
%\Ptwo(\Aboxk{A}{n-1}{B})
%\,=\,
%\Ptwo(\Aboxk{A}{n}{B})
    \,.
\end{equation}
Recalling that
    $A\boxtimes B$ is a subset of $A\times B$, we can 
    rewrite~\eqref{refvv} as
    \begin{equation}
      \label{rewxc}
  \Ptwo\bigl((A\times B)\setminus(A\boxtimes B)\bigr)\,=\,0\,.
    \end{equation}
Since \(p\) belongs to $]0,1[$, the product measure \(\Ptwo\) has full support on
\(\Omega^2\). Hence the equality~\eqref{rewxc} implies that
  $(A\times B)\setminus(A\boxtimes B)$ is void,
%\[
  %(A\times B)\setminus(A\boxtimes B)\,=\,\varnothing\,,
%\]
which is exactly the condition of the theorem.
%It remains to obtain the second condition. 
%Using the 
%identity~\eqref{secgap}, 
%and noticing that 
%\(x*_n y=y\), 
%we see that the 
%second equality~\eqref{refww} is equivalent to 
%\begin{equation}
%\label{recondkn} 
%p(1-p)\,\Ptwo
  %\Bigg(
  %\Bigg\{
%\begin{matrix}
    %(x,y)\in\Omega\times\Omega:
%n \text{ is a symmetric disjoint}\\
%\text{pivot for } (A,B) \text{ in } (x,y)
%\end{matrix}
  %\Bigg\}
  %\Bigg)\,=\,0\,.
%\end{equation}
%
%This means that there is no configuration 
%in $A\times B$ for which the coordinate $n$ is 
%a double symmetric pivot.
%the preceding
%argument gives
%
%The second condition follows.

Conversely, let us consider two increasing events $A,B$ satisfying the 
condition of theorem~\ref{thrm1}.
%, i.e.,
%\[
  %(A\times B)\setminus(A\boxtimes B)\,=\,\varnothing\,,
%\]
%so 
We apply lemma~\ref{suprilem} and we obtain
that 
%\begin{equation}
  %\label{lasrzz}
    %\Ptwo(A\boxtimes B) 
%\,=\, \Ptwo(A\times B)
    %\,=\, P(A)P(B)
%\,.
%\end{equation}
%The second condition 
%says that  
\begin{multline}
  \forall k\in\unn\quad
\forall
    (x,y)\in\Omega\times\Omega\cr
k \text{ is not a symmetric disjoint}
\text{ pivot for } (A,B) \text{ in } (x,y)\,.
\end{multline}
%\[
  %Ptwo\bigl(\SDP_i(x,y)\bigr)\,=\,0
  %qquad\text{for every } i\in\unn\,.
%]
%This implies that \(\SDP_i(u,v)\) never occurs, for any \((u,v)\in\Omega^2\). 
%never occurs at the hybrid pairs \((x,x*_k y)\). Thus the second term in
This implies furthermore that
\begin{equation}
  \forall k\in\unn\qquad
  \Bigg\{
\begin{matrix}
    (x,y)\in\Omega\times\Omega:
k \text{ is a symmetric disjoint}\\
\text{pivot for } (A,B) \text{ in } (x,x*_k y)
\end{matrix}
  \Bigg\}\,=\,\varnothing\,.
\end{equation}
Using the 
identity~\eqref{secgap}, we see that
\begin{equation}
  \label{saefgsecgap}
  \forall k\in\unn\qquad
\Ptwo(\Aboxk{A}{k}{B})
 \,=\, \Ptwo(\Aboxk{A}{k-1}{B}) 
%\,=\, 0
\,.
%\Ptwo\Big(k \text{ is a symmetric disjoint pivot for } (A,B) \text{ in } (x,x*_k y)\Big)\,.
\end{equation}
Therefore all the inequalities in the chain of 
inequalities~\eqref{chainine} are in fact equalities
and we have indeed equality in the BK inequality.
This proves the converse implication.
\bigskip

\noindent
{\bf AI.} 
No AI was used nor involved in the mathematics presented in this work.

%We finally prove the lower bound on the first gap.
%
%\begin{proof}[Proof of Proposition~\ref{lbi}]
%If a coordinate \(i\) is pivotal for \(A\) in \(x\), then every witness for \(A\) in \(x\) must contain \(i\). Likewise, if \(i\) is pivotal for \(B\) in \(y\), then every witness for \(B\) in \(y\) must contain \(i\). Hence, on the event
%\[
%\bigl\{\,(x,y)\in A\times B : \cP(A,x)\cap \cP(B,y)\,\ne\,\varnothing\,\bigr\}\,,
%\]
%every witness for \(A\) intersects every witness for \(B\). 
%Thus the above event is included in
%\((A\times B)\setminus A \boxtimes B\).
%\end{proof}
%\noindent
\bibliographystyle{amsplain}
\bibliography{abk.bib}

\end{document}